\documentclass[12pt]{amsart}

\usepackage{cancel}
\usepackage{latexsym, amssymb, amsmath}
\usepackage{soul}
\usepackage{amsfonts, graphicx}
\usepackage{graphicx,color}
\newcommand{\be}{\begin{equation}}
\newcommand{\ee}{\end{equation}}
\newcommand{\beq}{\begin{eqnarray}}
\newcommand{\eeq}{\end{eqnarray}}

\newtheorem{thm}{Theorem}[section]

\newtheorem{lma}{Lemma}[section]

\newtheorem{cor}{Corollary}[section]
\newtheorem{defn}{Definition}[section]
\theoremstyle{remark}

\numberwithin{equation}{section}

\def\dps{\displaystyle}

\def\be{\begin{equation}}
\def\ee{\end{equation}}
\def\bee{\begin{equation*}}
\def\eee{\end{equation*}}

\def\lf{\left}
\def\ri{\right}

\def\avint{\mathop{\ooalign{$\int$\cr$-$}}}
\def\K{K\"ahler }
\def\KR{K\"ahler-Ricci }
\def\Ric{\text{\rm Ric}}
\def\Rm{\text{\rm Rm}}

\def\dbar{\bar\p}
\def\wh{\widehat}
\def\wt{\widetilde}

\def\p{\partial}
\def\ddbar{\partial\bar\partial}

\def\heat{\lf(\frac{\p}{\p t}-\Delta\ri)}
\def\tr{\operatorname{tr}}

\def\e{\epsilon}
\def\a{{\alpha}}

\def\ijb{{i\bar{j}}}
\def\ii{\sqrt{-1}}
\def\R{\mathbb{R}}
\def\C{\mathbb{C}}
\begin{document}

\title[]
{Longtime existence of K\"ahler Ricci flow and holomorphic sectional curvature}

 \author{Shaochuang Huang}
\address[Shaochuang Huang]{Yau Mathematical Sciences Center, Tsinghua University, Beijing, China.}
\email{schuang@mail.tsinghua.edu.cn}

 \author{Man-Chun Lee}
\address[Man-Chun Lee]{Department of
 Mathematics, The Chinese University of Hong Kong, Shatin, Hong Kong, China.}
\email{mclee@math.cuhk.edu.hk}

\author{Luen-Fai Tam$^1$}
\address[Luen-Fai Tam]{The Institute of Mathematical Sciences and Department of
 Mathematics, The Chinese University of Hong Kong, Shatin, Hong Kong, China.}
 \email{lftam@math.cuhk.edu.hk}
\thanks{$^1$Research partially supported by Hong Kong RGC General Research Fund \#CUHK 14301517}

\author{Freid Tong}
\address[Freid Tong]{Department of Mathematics, Columbia University, 2990 Broadway, New York, NY 10027}
 \email{tong@math.columbia.edu}

\renewcommand{\subjclassname}{
  \textup{2010} Mathematics Subject Classification}

\date{May, 2018}

\begin{abstract}
In this work, we obtain a existence criteria for the longtime \KR flow solution. Using the existence result, we generalize a result by Wu-Yau on the existence of \K Einstein metric to the case with possibly unbounded curvature. Moreover, the \K Einstein metric with negative scalar curvture must be unique up to scaling.
\end{abstract}


\maketitle

\markboth{Shaochuang Huang, Man-Chun Lee, Luen-Fai Tam and Freid Tong}{Longtime existence of \KR flow and holomorphic sectional curvature}
\section{introduction}

In this work we will discuss the existence of K\"ahler-Einstein metric on a complete noncompact \K manifold in terms of upper bound of holomorphic sectional curvature. In \cite{WuYau2016}, Wu and Yau proved that if a compact complex manifold supports a \K metric with negative holomorphic sectional curvature, then it also supports a K\"ahler-Einstein metric with negative scalar curvature, under an additional assumption that the manifold is projective. Later Tosatti and Yang \cite{TosattiYang2015} were able to remove the assumption of projectivity. Using \KR flow, Normura \cite{Nomura2017} recovers the result by proving that under the assumption that the holomorphic sectional curvature is bounded above by a negative constant, the metric can be deformed under the normalized \KR flow to a K\"ahler-Einstein metric with negative scalar curvature.
In case that the holomorphic sectional curvature is quasi-negative, namely it is nonpositive and is negative somewhere, Diverio-Trapani \cite{DiverioTrapani2016} and Wu-Yau \cite{WuYau2016-2} then the canonical bundle is ample.

For the noncompact case,   it was proved by Wu and Yau \cite{WuYau2017} if a noncompact complex manifold supports a \K metric with holomorphic sectional curvature bounded between two negative constants, then it also supports a K\"ahler-Einstein metric with negative scalar curvature. It is well-known that if the holomorphic sectional curvature is bounded then the curvature is bounded. Hence one may consider using \KR flow to obtain the same result. This has been done successfully by the fourth author of this work.

In this paper, we first will give a rather general condition for a normalized \KR flow to converge to a K\"ahler-Einstein metric.  We prove the following:

\begin{thm}\label{t-intro-1}
Suppose there is a complete noncompact Hermitian metric $h$ on a complex manifold $M^n$ compatible with the complex structure $J$ such that
\begin{align}\label{Hneg}
H_h+\frac{2n}{n+1}|\wh\nabla_{\dbar} \wh T|_h\leq -k
\end{align}
for some $k>0$. Then any longtime solution of normalized \KR flow $g(t)$ will converge uniformly in $C^\infty$ to a \K Einstein metric  $g_\infty=-\Ric(g_\infty)$. {In particular, there is no Ricci flat \K metric on $M$ compatible with the same complex structure $J$.}
\end{thm}
Here $\wh\nabla$ is the derivative with respect to the Chern connection of $h$. See \cite{TosattiWeinkove2015} for more details on the Chern connection, its torsion and curvature.
Note that the K\"ahler-Einstein metric is unique, see \cite[p.30]{WuYau2017} for example.

By the theorem, the question is to obtain longtime solution to the \KR flow. We have:

\begin{thm}\label{t-intro-2}
Let $(M^n,g_0)$ be a complete \K manifold and $h$ is a fixed complete Hermitian metric on $M$ such that the following holds.
\begin{enumerate}
\item[  (i)] There exists smooth exhaustion $\rho\geq 1$ such that
$$\limsup_{\rho\rightarrow \infty}\left(\frac{|\partial \rho|_h}{\rho}+\frac{|\sqrt{-1}\partial\bar\partial \rho|_h}{\rho} \right)=0;$$
\item[\bf (ii)]  and the holomorphic sectional curvature of $h$ and torsion of $h$ satisfy
$$H_h+\frac{2n}{n+1}|\wh\nabla_{\bar\partial} \wh T|_h\leq -k $$ with $k\ge0$
\item[ (iii)]  $\exists \a>1$ such that on $M$,
$\a^{-1}g_0\leq h\leq \a g_0$, $|\wh T|_h\leq \a$;
\item [(iv)]
$\limsup_{\rho\rightarrow \infty}\rho^{-1}|\wh\nabla g_0|_h =0;$

\end{enumerate}
Then there is $\beta(n,\a)>0$ such that the \K Ricci flow has a complete solution $g(t)$ on $M\times [0,+\infty)$ with $g(0)=g_0$ and satisfies
 $$\beta h\leq g(t)$$ on $M\times [0,+\infty)$.
\end{thm}
It is known that if $M$ has bounded curvature, then it will support an exhaustion function $\rho$ with bounded gradient and Hessian \cite{Shi1989a,Tam2010}. Hence  $h$ is uniformly equivalent to a complete manifold with bounded curvature then condition (i) in the theorem is also satisfied. See also a recent result in  \cite{Huang2018}. Hence condition (i) is more general than the condition that the curvature is bounded.

Combining Theorems \ref{t-intro-1} and \ref{t-intro-2}, we conclude that if $(M^n,g_0)$ is a complete \K metric, then the normalized \KR flow will converge to the K\"ahler-Einstein metric with negative scalar curvature in the following cases:

\begin{itemize}
  \item [(a)] The holomorphic sectional curvature is bounded above by $-k$ for some $k>0$ and $g_0$ support an exhaustion function with bounded gradient and bounded complex Hessian.
  \item [(b)] There exists  a complete Hermitian metric $h$ so that $g_0, h$ satisfy the conditions in Theorem \ref{t-intro-2} with $k>0$.
  \item [(c)] $g_0$ satisfy a Sobolev inequality and the curvature is bounded in some $L^p$ sense so that the holomorphic sectional curvature is bounded above by $-k$ for some $k>0$. (See more precise statement in Corollary \ref{c-Xu}.)
\end{itemize}
In case $g_0$ has bounded curvature so that the holomorphic sectional curvature is bounded above by $-k<0$ for some constant $k$, then the conditions in (c) will also be satisfied. (a)--(c) are some generalizations to Wu-Yau's result \cite{WuYau2017}.

{\it Acknowledgement}: The second author would like to thank Professor Fangyang Zheng for answering his question. Part of the works was done when the second author visited the Yau Mathematical Sciences Center of Tsinghua University in Beijing, which he would like to thank for the hospitality.

\section{Short time existence lemma}\label{s-shorttime}
As a approximation, it is more convenient to consider the case when $g_0$ is only Hermitian metric but not necessary K\"ahler. Let $(M^n,g_0)$ be a complete noncompact Hermitian manifold with complex dimension $n$. In the following, connection and  curvature will be refered to the Chern connection and curvature with respect to the Chern connection. When the torsion vanishes, it coincides with the Levi-Civita connection. {For basic facts on Chern connection and curvature of Hermitian manifolds see \cite{TosattiWeinkove2015} for example}.  In this section, we want to discuss the existence of the Chern-Ricci flow:
\be\label{e-CRflow-1}
\left\{
  \begin{array}{ll}
    \frac{\p}{\p t}g_\ijb=  &  -R_\ijb; \\
    g(0)=  &g_0.
  \end{array}
\right.
\ee
where $R_\ijb=-\p_i\p_{\bar j}\log \det (g(t))$ is the Chern-Ricci curvature of $g(t)$.
This equation is equivalent to the following parabolic complex Monge-Amp\`ere equation:
\be\label{e-MA-1}
\left\{
  \begin{array}{ll}
    \frac{\p}{\p t}\psi=  &  \dps{\log \frac{(\omega_0-t\Ric(\omega_0)+\ii\p\bar\p \psi)^n}{\omega_0^n} }; \\
    \psi(0)=  &0.
  \end{array}
\right.
\ee
More precisely, if $g(t)$ is a solution to \eqref{e-CRflow-1}, let
\be\label{e-psi}
\psi(x,t)=\int_0^t\log \lf(\frac {\omega^n(x,s)}{\omega_0^n(x)}\ri)ds
\ee
where $\omega(t)$ and $\omega_0$ are the associated (1,1) forms of $g(t)$, $g_0$ respectively. Then $\psi$ satisfies \eqref{e-MA-1}. One can see that
  $\omega(t)= \omega_0-t\Ric(\omega_0)+\ii\p\bar\p \psi.$ Conversely, if $\psi$ is a smooth solution to \eqref{e-MA-1} so that $\omega_0-t\Ric(\omega_0)+\ii\p\bar\p \psi>0$, then $\omega(t)$ defined by the above relation satisfies \eqref{e-CRflow-1}. We will say that $\psi$ is the solution of \eqref{e-MA-1} corresponding to the solution $g(t)$ of \eqref{e-CRflow-1}.

In \cite{LeeTam2017}, it has been shown that when $(M,g_0)$ has bounded curvature of infinity order, the Monge-Amp\`ere equation \eqref{e-MA-1} and hence the Chern-Ricci flow equation \eqref{e-CRflow-1} has a short time solution on $M$.
\begin{lma}[see \cite{ChauTam2011,LeeTam2017}]\label{l-shortime-1} Let $(M^n,g_0)$ be a complete noncompact Hermitian metric. Suppose $g_0$ has bounded geometry of infinite order, then \eqref{e-CRflow-1} has a solution $g(t)$ on $M\times[0,S]$ for some $S>0$ and there is a constant $C>0$ such that
$C^{-1}g_0\le g(t)\le Cg_0$.
\end{lma}

\section{A-priori estimate for the Chern Ricci flow}\label{apesti}
Let $(M_n,g)$ be a Hermitian manifold. Under a local holomorphic coordinate
system $(z_1, . . . , z_n)$, the torsion tensor of $g$ is defined by
$$T_{ij\bar l}= \partial_i g_{j\bar l}-\partial_j g_{i\bar l}.$$
{Let  $T_{ij}^k=g^{k\bar l}T_{ij \bar l}$, then $T_{ij}^k=\Gamma_{ij}^k-\Gamma_{ji}^k$ where $\Gamma_{ij}^k$ is the Chern connection. $T_{ij}^k$ is usually called the torsion. Here we use $\Gamma_{ijk}$ to denote the  torsion. The advantage is that it is invariant under the Chern-Ricci flow. }   The curvature tensor of the Chern connection has components
$$R_{i\bar j k\bar l}=-\frac{\partial^2g_{k\bar l}}{\partial z_i\partial \bar  z_j}+g^{q\bar p} \frac{\partial g_{k\bar p}}{\partial z_i}\frac{\partial g_{q\bar l}}{\partial \bar z_j}.$$

If the torsion tensor $T= 0$, then $g$ is K\"ahler. It can be checked easily that for $X,Y\in T^{1,0}M$, $R(X,\bar X, Y,\bar Y)$ is real-valued. We introduce the following curvature condition.
\begin{defn}
We say that $(M,g)$ has holomorphic sectional curvature bounded above by $\kappa$ if for any $p\in M$, $X\in T^{1,0}_pM$,
$$R(X,\bar X,X,\bar X)\leq \kappa |X|^4.$$
For notational convenience, we denote it to be $H_g\leq \kappa$.
\end{defn}

Let $g(t)$ be a solution of the Chern-Ricci flow with initial metric $g(0)=g_0$ and $h$ is another Hermitian metric on $M$. Now we wish to obtain some a priori estimates for $g(t)$. First we list some evolution equations which is related to the Chern-Ricci flow.

\begin{lma}\label{l-tracee}
The evolution equation of $\Lambda=tr_gh =g^{i\bar j }h_{i\bar j}$ is given by
\begin{equation}
\begin{split}
\heat \Lambda &= (I)+(II)+(III).
\end{split}
\end{equation}
where
\begin{align*}
(I)&=- h_{k\bar l}g^{i\bar j} g^{p\bar q} \Psi_{pi}^k \overline{\Psi_{qj}^l} +2Re\left[g^{i\bar j}g^{k\bar l} g^{p\bar q} h_{k\bar j}\Psi_{\bar l \bar q}^{\bar s}(T_0)_{ip\bar s}\right];\\
(II)&=g^{l\bar k}g^{j\bar i} g^{q\bar p} h_{j\bar k}(T_0)_{\bar i\bar p r}\left[ \hat T_{ql \bar s} h^{r \bar s}-(T_0)_{l q \bar s}g^{r\bar s}\right]\\
&\quad +g^{i\bar j}g^{k\bar l} g^{p\bar q} h_{k\bar j}\left[\hat\nabla_p (T_0)_{\bar q\bar l i}+\hat\nabla_{\bar l} (T_0)_{pi\bar q}\right];\\
(III)&=g^{i\bar j} g^{p\bar q} \hat R_{p\bar q i\bar j}.
\end{align*}
In particular, $(I)\leq h_{p\bar r} h_{c\bar q}h^{k\bar a}g^{s\bar r} g^{c\bar d} g^{i\bar j} g^{p\bar q}(T_0)_{si\bar a}(T_0)_{\bar d \bar j k}$. Moreover, we have the following evolution equation of the quantity $\log tr_gh$.

\begin{align}\label{l-trace}
 \heat \log \Lambda
&= (IV)+\Lambda^{-1}\Big[(II)+(III) \Big]
\end{align}
where
\begin{align*}
(IV)&\leq \Lambda^{-1}h_{p\bar r} h_{c\bar q}h^{k\bar a}g^{s\bar r} g^{c\bar d} g^{i\bar j} g^{p\bar q}(T_0)_{si\bar a}(T_0)_{\bar d \bar j k}\\
&\quad +2\Lambda^{-2}Re\Big[h_{p\bar r} g^{a\bar r} g^{i\bar l} g^{p\bar q}  (T_0)_{ai\bar l} \partial_{\bar q} \Lambda\Big].
\end{align*}

Here $\Delta F=g^{i\bar j}\partial_i \partial_{\bar j}F$ for function $F\in C^2(M)$. $T_0$ and $\hat T$ are the torsion of metric $g_0$ and $h$ respectively.
\end{lma}
\begin{proof}
\begin{align*}
\partial_t tr_g  h= g^{i\bar q} g^{p\bar j}  h_{i\bar j} R_{p\bar q}.
\end{align*}
Denote $\Psi_{ij}^k=\hat \Gamma_{ij}^k-\Gamma_{ij}^k$.
\begin{align*}
\Delta tr_g h &=g^{i\bar j}g^{p\bar q}\nabla_ {\bar q}\nabla_p   h_{i\bar j}\\
&=g^{i\bar j}g^{p\bar q} \nabla_{\bar q} \left( \Psi_{pi}^k  h_{k\bar j}\right)\\
&=g^{i\bar j} g^{p\bar q}  \left[ \left(R_{p\bar q i}^k-\hat R_{p\bar q i}^k \right)  h_{k\bar j}+\Psi_{pi}^k \Psi_{\bar q \bar j}^{\bar l} h_{k\bar l}\right].
\end{align*}

Using the fact that  $\partial \theta_0=\partial \theta$, we have
\begin{align*}
R_{p\bar q i\bar l}&=R_{i\bar lp\bar q}-\nabla_pT_{\bar q\bar li}-\nabla_{\bar l}T_{pi\bar q}\\
&=R_{i\bar lp\bar q}-\nabla_p (T_0)_{\bar q\bar li}-\nabla_{\bar l} (T_0)_{pi\bar q}.
\end{align*}

Hence,
\begin{align*}
g^{i\bar j} g^{p\bar q}  h_{k\bar j}R^k_{p\bar q i}
&= g^{i\bar j}g^{k\bar l} g^{p\bar q}  h_{k\bar j}R_{p\bar q i \bar l}\\
&=g^{i\bar j}g^{k\bar l} g^{p\bar q}  h_{k\bar j} \left[ R_{i\bar lp\bar q}-\nabla_p (T_0)_{\bar q\bar li}-\nabla_{\bar l} (T_0)_{pi\bar q}\right]\\
&=g^{i\bar j} g^{k\bar l} h_{k\bar j}R_{i\bar l}-g^{i\bar j}g^{k\bar l} g^{p\bar q}  h_{k\bar j} \left[\nabla_p (T_0)_{\bar q\bar li}+\nabla_{\bar l} (T_0)_{pi\bar q}\right].
\end{align*}

Therefore,
\begin{align*}
&\quad \heat  tr_gh \\
&=- h_{k\bar l}g^{i\bar j} g^{p\bar q} \Psi_{pi}^k \overline{\Psi_{qj}^l}+g^{i\bar j} g^{p\bar q} \hat R_{p\bar q i\bar j}+g^{i\bar j}g^{k\bar l} g^{p\bar q} h_{k\bar j} \left[\nabla_p (T_0)_{\bar q\bar li}+\nabla_{\bar l} (T_0)_{pi\bar q}\right]\\
&=- h_{k\bar l}g^{i\bar j} g^{p\bar q} \Psi_{pi}^k \overline{\Psi_{qj}^l} +g^{i\bar j}g^{k\bar l} g^{p\bar q} h_{k\bar j}\left[ \Psi_{pi}^r (T_0)_{\bar l\bar q r}+\Psi_{\bar l \bar q}^{\bar s}(T_0)_{ip\bar s}\right]\\
&\quad +g^{i\bar j}g^{k\bar l} g^{p\bar q} h_{k\bar j}\left[\hat\nabla_p (T_0)_{\bar q\bar l i}+\hat\nabla_{\bar l} (T_0)_{pi\bar q}\right]+g^{i\bar j} g^{p\bar q} \hat R_{p\bar q i\bar j}\\
&=- h_{k\bar l}g^{i\bar j} g^{p\bar q} \Psi_{pi}^k \overline{\Psi_{qj}^l} +2Re\left[g^{i\bar j}g^{k\bar l} g^{p\bar q} h_{k\bar j}\Psi_{\bar l \bar q}^{\bar s}(T_0)_{ip\bar s}\right]\\
&\quad +g^{l\bar k}g^{j\bar i} g^{q\bar p} h_{j\bar k}(T_0)_{\bar i\bar p r}\left[ \hat T_{ql \bar s} h^{r \bar s}-(T_0)_{l q \bar s}g^{r\bar s}\right]\\
&\quad +g^{i\bar j}g^{k\bar l} g^{p\bar q} h_{k\bar j}\left[\hat\nabla_p (T_0)_{\bar q\bar l i}+\hat\nabla_{\bar l} (T_0)_{pi\bar q}\right]+g^{i\bar j} g^{p\bar q} \hat R_{p\bar q i\bar j}\\
&=(I)+(II)+(III),
\end{align*}
where
\begin{align*}
(I)&=- h_{k\bar l}g^{i\bar j} g^{p\bar q} \Psi_{pi}^k \overline{\Psi_{qj}^l} +2Re\left[g^{i\bar j}g^{k\bar l} g^{p\bar q} h_{k\bar j}\Psi_{\bar l \bar q}^{\bar s}(T_0)_{ip\bar s}\right];\\
(II)&=g^{l\bar k}g^{j\bar i} g^{q\bar p} h_{j\bar k}(T_0)_{\bar i\bar p r}\left[\hat T_{ql \bar s} h^{r \bar s}-(T_0)_{l q \bar s}g^{r\bar s}\right]\\
&\quad +g^{i\bar j}g^{k\bar l} g^{p\bar q} h_{k\bar j}\left[\hat\nabla_p (T_0)_{\bar q\bar l i}+\hat\nabla_{\bar l} (T_0)_{pi\bar q}\right];\\
(III)&=g^{i\bar j} g^{p\bar q} \hat R_{p\bar q i\bar j}.
\end{align*}
Thus,
\begin{align*}
 \heat \log \Lambda
&= {\Lambda}^{-1}\heat \Lambda +\Lambda^{-2}g^{i\bar j} \partial_i \Lambda \, \partial_{\bar j} \Lambda\\
&= \frac{1}{\Lambda}\left[  (I)+\frac{1}{tr_gh }|\partial \Lambda|^2\right] +\frac{1}{\Lambda}\Big[(II)+(III)\Big].
\end{align*}

In the special case that $g_0$ is K\"ahler, it was shown by Yau \cite{Yau1978} that the first bracket term is nonpositive. In the Hermitian case, we follow a generalization of this argument in \cite{TosattiWeinkove2015}. We consider the following nonegative term.
\begin{align*}
K&=h_{k\bar l}g^{i\bar j} g^{p\bar q}\left(\Psi_{pi}^k-\frac{\delta^k_i}{tr_gh}\partial_p \Lambda+C^k_{pi}\right)\left(\Psi_{\bar q\bar j}^{\bar l}-\frac{\delta_{\bar j}^{\bar l}}{tr_gh}\partial_{\bar q}\Lambda+C^{\bar l}_{\bar q \bar j}\right)\\
&=h_{k\bar l}g^{i\bar j} g^{p\bar q}\Psi_{pi}^k\Psi_{\bar q\bar j}^{\bar l}
-\frac{1}{tr_gh}|\partial tr_gh|^2+h_{k\bar l} g^{i\bar j} g^{p\bar q} C^k_{pi} \left( \Psi_{\bar q\bar j}^{\bar l}-\frac{\delta_{\bar j}^{\bar l}}{tr_gh}\partial_{\bar q}\Lambda\right)\\
&\quad +h_{k\bar l} g^{i\bar j} g^{p\bar q}  C^{\bar l}_{\bar q \bar j}\left(\Psi_{pi}^k-\frac{\delta^k_i}{tr_gh}\partial_p \Lambda \right)+h_{k\bar l} g^{i\bar j} g^{p\bar q} C^k_{pi} C^{\bar l}_{\bar q \bar j}.
\end{align*}
Hence,
\begin{align*}
 (I)+\frac{1}{\Lambda}|\partial \Lambda|^2
&=-K+2Re\left[g^{i\bar j}g^{k\bar l} g^{p\bar q} h_{k\bar j}\Psi_{\bar l \bar q}^{\bar s}(T_0)_{ip\bar s}\right]\\
&\quad +2Re\left[h_{k\bar l} g^{i\bar j} g^{p\bar q} C^k_{pi}  \Psi_{\bar q\bar j}^{\bar l}\right]\\
&\quad +h_{k\bar l} g^{i\bar j} g^{p\bar q} C^k_{pi} C^{\bar l}_{\bar q \bar j}-2\Lambda^{-1}Re\left[h_{k\bar l} g^{i\bar l} g^{p\bar q} C^k_{pi}\partial_{\bar q} \Lambda\right]\\
&=-K +h_{k\bar l} g^{i\bar j} g^{p\bar q} C^k_{pi} C^{\bar l}_{\bar q \bar j}-2\Lambda^{-1}Re\left[h_{k\bar l} g^{i\bar l} g^{p\bar q} C^k_{pi}\partial_{\bar q} \Lambda\right]\\
&\quad +2Re \left[\Psi^{\bar l}_{\bar q \bar j} g^{p\bar q}g^{i\bar j}\big(C^k_{pi}h_{k\bar l}+g^{k\bar r}h_{p\bar r}(T_0)_{ki\bar l}\big)\right].
\end{align*}
Therefore, if we choose the tensor $C$ to be
$$C_{pi}^q=-g^{k\bar r}h^{q\bar l}h_{p\bar r} (T_0)_{ki\bar l}$$
then the last term vanished. Hence, we have
\begin{align*}
 \heat \log \Lambda
&= {\Lambda}^{-1}\heat \Lambda +\Lambda^{-2}g^{i\bar j} \partial_i \Lambda \, \partial_{\bar j} \Lambda\\
&= (IV)+\Lambda^{-1}(II)+\Lambda^{-1}(III).
\end{align*}
where
\begin{align*}
(IV)&\leq \Lambda^{-1}h_{p\bar r} h_{c\bar q}h^{k\bar a}g^{s\bar r} g^{c\bar d} g^{i\bar j} g^{p\bar q}(T_0)_{si\bar a}(T_0)_{\bar d \bar j k}\\
&\quad +2\Lambda^{-2}Re\Big[h_{p\bar r} g^{a\bar r} g^{i\bar l} g^{p\bar q}  (T_0)_{ai\bar l} \partial_{\bar q} \Lambda\Big].
\end{align*}
The inequality on $(I)$ follows the same line as in the inequality of $A_1$ by considering a simpler quantity
$$K=h_{k\bar l} g^{i\bar j} g^{p\bar q}(\Psi_{pi}^k -C_{pi}^k)(\Psi_{\bar q\bar j }^{\bar l}-C_{\bar q\bar j}^{\bar l}).$$
\end{proof}

We have the following form of Royden's Lemma \cite{Royden1980} which relates the holomorphic sectional curvature with a bisectional curvature quantity.

In the following, $|\wh \nabla_{\bar\p}\wh T|_h(x)$ at a point $x$ is defined as:
\be\label{e-dbar}
|\wh \nabla_{\bar\p}\wh T|_h =\max |\wh\nabla_{\bar i}\wh T_{jl\bar k}|
\ee
where the maximum is taken over all   unitary frames $e_i$ of $h$ at $x$. Define $|\wh\nabla_{\dbar} T_0|_h$ similarly.

\begin{lma}\label{royden}
Let $(M,h)$ be a Hermitian manifold and $g$ is another metric on $M$. Suppose that the holomorphic sectional curvature of $h$ at $x$ is bounded above by $\kappa(x)$. Suppose $\kappa(x)\le \kappa_0$. Then we have

$$ g^{i\bar j} g^{k\bar l}  \hat R_{i\bar jk\bar l}\leq \lf(\frac{n+1}{2n}\kappa+|\wh \nabla_{\bar\p} \wh T|_h\ri)(\tr_gh)^2 +\frac12\kappa_0 \lf[-\frac1n(\tr_gh)^2+g^\ijb g^{k\bar l}h_{kj}h_{i\bar l}\ri].$$
\end{lma}
\begin{proof}
Following the proof in \cite{Royden1980} without appealing the symmetry of $\hat R$, we can deduce that at $x$,
$$g^{i\bar j}g^{k\bar l}\hat R_{i\bar j k\bar l}+g^{i\bar j} g^{k\bar l} \hat R_{i\bar l k\bar j}\leq \kappa(tr_g h)^2+\kappa g^{i\bar j} g^{k\bar l}h_{k\bar j} h_{i\bar l}.$$
By the "\K" identity for the Chern curvature, e.g. see \cite{TosattiWeinkove2015}. We have
\begin{equation}
\begin{split}
g^{i\bar j} g^{k\bar l}  \hat R_{i\bar jk\bar l}&=g^{i\bar j}g^{k\bar l}(\hat R_{i\bar l k\bar j}+\hat\nabla_i \hat T_{\bar j \bar l k})\\
&=\frac{1}{2}g^{i\bar j}g^{k\bar l}(\hat R_{i\bar l k\bar j}+\hat R_{i\bar j k\bar l}+\hat\nabla_i \hat T_{\bar j \bar l k})\\
&\leq \frac{\kappa}{2}(tr_g h)^2+ g^{i\bar j} g^{k\bar l}\left(\frac{\kappa}{2}h_{k\bar j} h_{i\bar l}+\hat\nabla_i \hat T_{\bar j \bar l k}\right)\\
  &\le\frac12(\kappa(x)-\kappa_0)\lf[(\tr_gh)^2+g^\ijb g^{k\bar l}h_{kj}h_{i\bar l}\ri]+(\tr_gh)^2|\wh \nabla_{\bar\p} \wh T|_h\\
  &+\frac12\kappa_0 \lf[(\tr_gh)^2+g^\ijb g^{k\bar l}h_{kj}h_{i\bar l}\ri]\\
  &\le \frac12(\kappa(x)-\kappa_0)(1+\frac1n)(\tr_gh)^2+(\tr_gh)^2|\wh \nabla_{\bar\p} \wh T|_h\\
  &+\frac12\kappa_0 \lf[(\tr_gh)^2+g^\ijb g^{k\bar l}h_{kj}h_{i\bar l}\ri]\\
  &= \lf(\frac{n+1}{2n}\kappa+|\wh \nabla_{\bar\p} \wh T|_h\ri)(\tr_gh)^2 +\frac12\kappa_0 \lf[-\frac1n(\tr_gh)^2+g^\ijb g^{k\bar l}h_{kj}h_{i\bar l}\ri]
\end{split}
\end{equation}
\end{proof}

\begin{cor}\label{c-trace}
With the same assumptions and notation as in Lemma \ref{l-tracee}. Suppose the holomorphic sectional curvature of $h$ is bounded above by $\kappa(x)$ at $x$ and suppose $\frac{n+1}{2n}\kappa(x)+|\wh\nabla_{\dbar}\wh T|(x)\le \kappa_0$ for some $\kappa_0\ge0$ for all $x$.
  Then
\bee
\heat \Lambda\le c(n)\lf(\Lambda^4|T_0|^2_h+\Lambda^3(|T_0|_h|\wh T|_h+|\wh \nabla_{\dbar} T_0|_h)+\Lambda^2 \kappa_0\ri)
\eee
  for some constant $c(n)>0$ depending only on $n$.

\end{cor}

To get a $C^0$ estimate, it is useful to consider the Chern scalar curvature of $g(t)$ which gives us information on the derivatives of the volume form.
\begin{lma}\label{scalar}Under the evolution of metrics
$$\partial_t g=-Ric$$
the Chern scalar curvature $R$ satisfies
$$\heat R=|Ric|^2 \geq \frac{1}{n}R^2.$$
\end{lma}
\begin{proof}
\begin{align*}
\partial_t R &=\partial_t (g^{i\bar j} R_{i\bar j}) \\
&=R^{i\bar j}R_{i\bar j}-g^{i\bar j} \partial_i\partial_{\bar j}(\partial_t \log\det g)\\
&=|Ric|^2+\Delta R.
\end{align*}
The inequality can be observed by taking a coordinate at $p$ such that $g_{i\bar j}=\delta_{ij}$ and $R_{i\bar j}=\lambda_i \delta_{ij}$. Then it follows immediately by Cauchy inequality.
\end{proof}

We have the following maximum principle.

\begin{lma}\label{max} Let $(M^n,h)$ be a complete noncompact Hermitian manifold satisfying condition: {\bf (a1)}There exists a  smooth positive real exhaustion function $\rho$  such that $|\partial \rho|^2_h+|\sqrt{-1}\partial\bar\partial \rho|_h\leq C_1$. Suppose $g_0$ is another Hermitian metric uniformly equivalent to $h$ and $g(t)$ is a solution to the Chern-Ricci flow with initial metric $g(0)=g_0$ on $M\times[0,S)$. Assume for any $0<S_1<S$, there is $C_2>0$ such that
$$
C_2^{-1}h\le g(t)
$$
for $0\leq t\le S_1$.
Let $f$ be  a smooth function    on $M\times[0,S)$ which is bounded from above such that
$$
\heat f\le0
$$
on $\{f>0\}$. Suppose  $f\le 0$ at $t=0$, then $f\le 0$ on $M\times[0,S)$.
\end{lma}
\begin{proof}
For any $\e>0$, if $\sup_{M\times [0,T]}(f-\e \rho-2\e C_1C_2t)>0$, then there is $(x_0,t_0)$ with $t_0>0$ such that $f-\e \rho-2\e C_1C_2t \leq 0$ on $M\times [0,t_0]$ and $f-\e \rho-2\e C_1C_2t=0$ at $(x_0,t_0)$. In particular, $f(x_0,t_0)>0$. Hence at $(x_0,t_0)$, we have
$$0\leq \heat (f-\e \rho-2\e C_1C_2t)<0,$$
which is impossible. Since $\e$ is arbitrary, this completes the proof.
\end{proof}


Next we give a local estimate on the Chern scalar curvature's lower bound.
\begin{lma}\label{locR}
Suppose $h$ is a fixed hermitian metric satisfying {\bf (a1)} and $g(t)$ is a solution to the Chern Ricci flow on $M\times [0,S]$ with $g(t) \geq \alpha^{-1} h$ for some $\alpha>1$. Then for any $0<r_1<r_2$, there exists $C_n>0$ such that for any $x\in M$ with $\rho(x)<r_1$ and $t\in [0,S]$, we have
$$R(x,t)\geq -\max\left\{C_n \alpha  [(r_2-r_1)^{-2}+1],\sup_{\rho(y)<r_2}R_-(y,0) \right\}.$$
\end{lma}
\begin{proof}
Let $\phi$ be a cutoff function on $\mathbb{R}$ such that $\phi\equiv 1$ on $(-\infty,1]$, vanishes outside $(-\infty,2]$ and satisfies $\phi^{-1}|\phi'|^2\leq 100$ and $\phi'' \geq -100\phi$. Define
$$\Phi(x,t)=\phi\left( \frac{\rho(x)+r_2-2r_2}{r_2-r_1}\right).$$
When the function $\Phi R$ achieves its local minimum at $(x_0,t_0)$ in which we may assume $R(x_0,t_0)<0$ and $t_0>0$, it satisfies the following.
\begin{align*}
0&\geq  \heat (\Phi R)\\
&=\Phi \heat R -R \Delta \Phi -2Re\left( g^{i\bar j}\partial_i \Phi \, \partial_{\bar j} R\right)\\
&\geq \frac{1}{n}\Phi R^2-R \left[\frac{\phi''}{(r_2-r_1)^2}|\partial \rho|^2+\frac{\phi'}{r_2-r_1}\Delta \rho-2\frac{(\phi')^2}{(r_2-r_1)^2\phi}|\partial\rho|^2\right]\\
&\geq  \frac{1}{n}\Phi R^2+C_n\Lambda R[(r_2-r_1)^{-2}+1].
\end{align*}
Hence, at its minimum point $(x_0,t_0)$,
$$\Phi R\geq -C_n \Lambda  [(r_2-r_1)^{-2}+1].$$
The conclusion follows by the minimum principle.
\end{proof}

\section{Existence of the Chern-Ricci flow}\label{existence}

In this section, we will discuss the existence of the Chern-Ricci flow starting from a Hermitian metric with holomorphic sectional curvature bounded from above. We will give a estimate on the existence time. More generally, we will consider initial metric which is uniformly equivalent to a Hermitian metric with holomorphic sectional curvature bounded from above.

Recall the following definition of bounded geometry:
\begin{defn}\label{boundedgeom} Let $(M^n, g)$ be a complete Hermitian manifold. Let  $k\ge 1$ be an integer and $0<\alpha<1$. $g$ is said to have   bounded geometry of   order $k+\alpha$ if there are positive numbers $r, \kappa_1, \kappa_2$  such that at every $p\in M$ there is a neighborhood $U_p$ of $p$, and local biholomorphism $\xi_{p}$ from $D(r)$ onto $U_p$ with $\xi_p(0)=p$ satisfying  the following properties:
  \begin{itemize}
    \item [(i)] the pull back metric $\xi_p^*(g)$   satisfies:
    $$
    \kappa_1 g_e\le \xi_p^*(g)\le \kappa_2 g_e$$ where $g_e$ is the standard metric on $\C^n$;
    \item [(ii)] the components $g_{\ijb}$ of $\xi_p^*(g)$ in the natural coordinate of $D(r)\subset \C^n$  are uniformly bounded in the standard $C^{k+\alpha}$ norm in $D(r)$ independent of $p$.
  \end{itemize}
$(M,g)$ is said to have bounded geometry of infinity order if instead of (ii) we have for any $k$, the $k$-th derivatives of $g_\ijb$ in $D(r)$ are bounded by a constant independent of $p$. $g$ is said to have bounded curvature of infinite order on a compact set $\Omega$ if (i) and (ii) are true for all $k$ for all $p\in \Omega$.
 \end{defn}

\begin{lma}\label{l-bgeom-1} Let $(M^n,g_0)$ be a Hermitian metric with bounded geometry of infinite order. Suppose $g_0$ is uniformly equivalent to a Hermitian metric $h$ with holomorphic sectional curvature and torsion satisfying: $H_h(x)$ bounded above by $\kappa(x)$ and $\frac{n+1}{2n}\kappa(x)+|\wh\nabla_{\dbar}\wh T|_h(x)\le \kappa_0$ for some $\kappa_0\ge0$ for all $x$,  so that
$$
 \a^{-1}h\le g_0\le \a h,
 $$
 Then the Chern-Ricci flow has a solution $g(t)$ with $g(0)=g_0$ on $M\times[0,S]$ with the following properties:

\begin{enumerate}
  \item [(i)]  There is a constant $c=c(n)>0$ so that
  $$
  S\ge\displaystyle{\frac1{  3c(n\a+1)^{3}   \mathfrak{s} }}=:S_1
  $$
   where
  $$
  \mathfrak{s}=\sup_M\lf( |T_0|^2_h+ |T_0|_h|\wh T|_h+|\wh\nabla_{\dbar} T_0|_h +\kappa_0\ri)
  $$
  where $T_0$ is the torsion of $g_0$, $\wh T$ is the torsion of $h$ and $\wh \nabla$ is the derivative of $h$ with respect to the Chern connection; and
  \item [(ii)] $g(t)$ is uniformly equivalent to $h$ with
  $$
  \tr_{g}h\le \lf(\frac{1}{(n\a+1)^{-3}-3c_2  \mathfrak{s} t}\ri)^\frac13
  $$
  on $M\times[0,S_1]$.
\end{enumerate}

\end{lma}
\begin{proof} By \cite[Theorem 4.2]{LeeTam2017}, there is  a maximal $S>0$ such that the Chern-Ricci flow has a solution $g(t)$ with $g(0)=g_0$ on $M\times[0,S)$ so that $g(t)$ is uniformly equivalent to $g_0$ on $[0,S']$ for all $S'<S$. Let $\Lambda=\tr_{g(t)}h$. By Corollary \ref{c-trace},

\bee
\begin{split}
\heat \Lambda\le& c_1 \lf(\Lambda^4|T_0|^2_h+\Lambda^3(|T_0|_h|\wh T|_h+|\wh T_0|) +\kappa_0)\ri)\\
\le& c_1\lf(\Lambda+1\ri)^4 \mathfrak{s}
\end{split}
\eee
 Here and below $c_i$ will denote a positive constants depending only on $n$.
 Let
 $$
 v(t)=\lf(\frac{1}{(n\a+1)^{-3}-3c_2  \mathfrak{s} t}\ri)^\frac13
 $$
 Then $v(t)$ is defined on $[0,S_1)$ with $S_1=1/ \lf[3c_2(n\a+1)^{3}   \mathfrak{s} \ri]  $, with
$$\frac{dv}{dt}=c_2\mathfrak{s}v^4$$
and $v(0)\ge (\Lambda+1)|_{t=0}$.
Suppose $S< S_1$. Since $\Lambda$ and $v$ are  bounded on $[0,S']$ for all $0<S'<S$, by Lemma \ref{max} as in the proof of \cite[Theorem 4.2]{LeeTam2017}, one can conclude that
$$
\Lambda\le v(t)-1
$$
on $M\times[0,S)$. In particular,
\be\label{e-trace-1}
h\le c_3(v(t)-1)g(t)
\ee

If $S<S_1$, then $v(t)\le C_1<\infty$ on $[0,S]$ for some $C_1$. Hence $\Lambda\le C_1$ on $M\times[0,S)$.

On the other hand, since $g_0$ has bounded geometry of infinite order, by Lemma \ref{locR}, we conclude that $R(x,t)\ge -C_2$ on $M\times[0,S)$ for some $C_2$. Since
$$
\frac{\p}{\p t}\lf(\log\frac{ \det(g(t) }{\det (h)}\ri)=-R\le C_2,
$$
we conclude that $\det(g(t))\le C_3\det(h)$. Together with \eqref{e-trace-1}, we conclude that
$$
C_3^{-1}g_0\le g(t)\le C_3g_0
$$
on $M\times[0,S)$ for some $C_3>0$. Here we have used the fact that $g_0$ is uniformly equivalent to $h$. Using the fact that $g_0$ has bounded geometry of infinite order and by the local estimates of \cite{ShermanWeinkove2013}, $g(t)$ can be extended to be a solution of the Chern-Ricci flow which is uniformly equivalent to $g_0$ beyond $S$. Hence we have $S\ge S_1$. This proves (i).

(ii) Follows from \eqref{e-trace-1}.

\end{proof}

Let $(M^n,h)$ be a complete noncompact Hermitian manifold and let $g_0$ be another Hermitian metric satisfying the following:
\begin{enumerate}
\item[ ({\bf a})] There exists smooth exhaustion $\rho\geq 1$, and constant $\beta>0$ such that
$$|\partial \rho|_h+|\sqrt{-1}\partial\bar\partial \rho|_h \leq \beta \rho  $$
if $\rho$ is large enough.
\item[({\bf b})] The holomorphic sectional curvature at $x$ is bounded from above by $\kappa(x)$,  and the torsion of $h$ is such that
    $$
    \frac{n+1}{2n}\kappa+| \wh\nabla_{\dbar}\wh T|_h\le \kappa_0
    $$
    for some $\kappa_0\ge0$.
    \end{enumerate}

\begin{thm}\label{t-existence}
Let  $(M^n,h)$ be a complete Hermitian metrics as above. Suppose $g_0$ is another Hermitian metric satisfying:

\begin{enumerate}
\item[(i)] $\a^{-1}g_0\leq h\leq \a g_0$
and  {$|\wh T|_{h} \leq \a$ for some $\a>0$;}
\item[(ii)]There is $\beta>0$,
$ |T_0|^2_{h}+|\wh  T|_h|T_0|_h+|\wh \nabla_{\dbar}(T_0)|_h\leq \beta;$ and
\item[(iii)]
$|\wh \nabla g_0|_h \leq \beta \rho $ for $\rho$ large enough.
\end{enumerate}

There exist constants    $c_1$ depending only on $n$ and $c_2$ depending only on $n,\a$  such that   there is a solution $g(t)$ for the Chern-Ricci flow on $M\times[0,S)$ with $g(0)=g_0$, where
$$
S=\frac{1}{2c_1(n\a+1)^3\mathfrak{s}}
$$
and $\mathfrak{s}=\kappa_0+c_2\beta(1+ \beta).$ Moreover,
$$
\tr_gh\le v(t)-1
$$
where
$$
v(t)=\lf(\frac1{(n\a+1)^{-3}-3c_1 \mathfrak{s}t}\ri)^\frac13.
$$
on $M\times[0,S)$.
\end{thm}

We want to apply Lemma \ref{l-bgeom-1}. However, in general it is not true that $g_0$ has bounded geometry of all order, we cannot apply Lemma \ref{l-shortime-1} directly to obtain a solution of the Chern-Ricci flow. We now proceed as in \cite{LeeTam2017,LeeTam2017-2} to construct a Hermitian approximation.

Let $\tau\in (0,\frac{1}{8})$, $f:[0,1)\to[0,\infty)$ be the function:
\be\label{e-exh-1}
 f(s)=\left\{
  \begin{array}{ll}
    0, & \hbox{$s\in[0,1-\tau]$;} \\
    -\displaystyle{\log \lf[1-\lf(\frac{ s-1+\tau}{\tau}\ri)^2\ri]}, & \hbox{$s\in (1-\tau,1)$.}
  \end{array}
\right.
\ee
Let   $\varphi\ge0$ be a smooth function on $\R$ such that $\varphi(s)=0$ if $s\le 1-\tau+\tau^2 $, $\varphi(s)=1$ for $s\ge 1-\tau+2 \tau^2 $
\be\label{e-exh-2}
 \varphi(s)=\left\{
  \begin{array}{ll}
    0, & \hbox{$s\in[0,1-\tau+\tau^2]$;} \\
    1, & \hbox{$s\in (1-\tau+2\tau^2,1)$.}
  \end{array}
\right.
\ee
such that $\displaystyle{\frac2{ \tau^2}}\ge\varphi'\ge0$. Define
 $$\mathfrak{F}(s):=\int_0^s\varphi(\tau)f'(\tau)d\tau.$$
From \cite{LeeTam2017-2}, we have:

\begin{lma}\label{l-exhaustion-1} Suppose   $0<\tau<\frac18$. Then the function $\mathfrak{F}\ge0$ defined above is smooth and satisfies the following:
\begin{enumerate}
  \item [(i)] $\mathfrak{F}(s)=0$ for $0\le s\le 1-\tau+\tau^2$.
  \item [(ii)] $\mathfrak{F}'\ge0$ and for any $k\ge 1$, $\exp( -k\mathfrak{F})\mathfrak{F}^{(k)}$ is uniformly  bounded.
  \item [(iii)]  For any $ 1-2\tau <s<1$, there is $\tau>0$ with $0<s -\tau<s +\tau<1$ such that
 \bee
 1\le \exp(\mathfrak{F}(s+\tau)-\mathfrak{F}(s-\tau))\le (1+c_2\tau);\ \ \tau\exp(\mathfrak{F}(s_0-\tau))\ge c_3\tau^2
 \eee
  for some absolute constants  $c_2>0, c_3>0$.
\end{enumerate}

\end{lma}

For   any  $\rho_0>0$, let  $U_{\rho_0}$ be the component of $ \{x|\ \rho(x)<\rho_0\}$ which contains a fixed point and $\rho$ is the positive exhaustion function mentioned above. Hence $U_{\rho_0}$ will exhaust $M$ as $\rho_0\to\infty$.

Let $\rho_i>1$ be a sequence increasing to $+\infty$, let $F^{(i)}(x)=\mathfrak{F}(\rho(x)/\rho_i)$. Let $g_{0,i}=e^{2F^{(i)}}g_0$. In the following, $F^{(i)}$ will be denoted simply by $F$ if there is no confusion arisen.

 Then $(U_{\rho_i},g_i)$ is a complete Hermitian metric, (e.g. see \cite{Hochard2016}) and $ g_{i,0}=g_0$ on $\{\rho(x)<(1-\kappa+\kappa^2)\rho_0\}$. Moreover, the new manifold has a very nice structure.
\begin{lma}[\cite{LeeTam2017}]\label{l-boundedgeom}
For each $\rho_i>1$ sufficiently large, $(U_{\rho_i},g_{0,i})$ has bounded geometry of infinite order.
\end{lma}

In the following, we will estimate the torsion and the holomorphic sectional curvature after performing conformal change.
\begin{lma}\label{conformal-t-h}Let $g_0$ and $h$ be as in Theorem \ref{t-existence}. For $i\to\infty$, let $g_{0,i}$ be as in Lemma \ref{l-boundedgeom} and $h_i=e^{2F}h$ for the corresponding $F$. Let $T_{0,i}$ be the torsion of $g_{0,i}$ Then there is a constant $c(n,\a)$ depending only on $n$ and $\a$ so that  as $i\to\infty$, there
 \begin{enumerate}
        \item [(i)] $|T_{0,i}|^2_{h_i}\le c\beta $, where
        \item [(ii)] $|T_{ 0,i }|_{h_i}|\hat  T^{(i)}|_{h_i}\le c\beta$;
        \item [(iii)] $|\wh \nabla^{(i)}_{\dbar} T_{0,i}|_{h_i}\le  c\beta(1+\beta)$, where $\wh \nabla^{(i)}$ is derivative with respect to the Chern connection of $h_i$;
            \item[(iv)]
            $$
\frac{n+1}{2n}\kappa_i(x)+|\wh\nabla^{(i)}_{\dbar} T_i |_{h_i}(x)\le \kappa_0
+c \beta(1+\beta)
$$
where $\kappa_i(x)$ is the upper bound of holomorphic sectional curvature of $h_i$ at $x$ and $T_i$ is the torsion of $h_i$.
      \end{enumerate}

\end{lma}
\begin{proof} In the following, $c_i$ will denote a positive constant depending only on $n, \a$. Note that we still have  $\a^{-1}g_{0,i}\leq h_i\leq \a g_{0,i};$. For notational convenience, we will denote $g=g_{0,i}$, $\tilde g=g_0$, $\hat h=h_i$ in the proof.

(i)
\begin{equation}
\begin{split}
(T_{0,i})_{p k\bar q }
&=\partial_p (e^{2F} g_{k\bar q})-\partial_k (e^{2F} g_{p\bar q})\\
&=2e^{2F} (F_p g_{k\bar q} -F_k g_{p\bar q}) +e^{2F}(T_{g})_{pk\bar q}\\
&=2(F_p \tilde g_{k\bar q}-F_k \tilde g_{p\bar q})+e^{2F} (T_{g})_{pk\bar q}\\
&= 2\rho_0^{-1}\mathfrak{F}' \left( \rho_p \tilde g_{k\bar q}-\rho_k \tilde g_{p\bar q}\right)+e^{2F} (T_{g})_{pk\bar q}.
\end{split}
\end{equation}

Hence
\bee
|T_{0,i}|^2_{h_i}\le c_1\beta.
\eee
This proves (i). The proof of (ii) is similar.

(iii)
\begin{equation}
\begin{split}
\widehat\nabla_{\bar l}^{(i)}(T_{0,i})_{p k\bar q }
&=\wh\nabla_{\bar l}^{(i)}\left(2(F_p  g_{k\bar q}-F_k  g_{p\bar q})+e^{2F} (T_{\tilde g})_{pk\bar q}\right)\\
&=2(F_{p\bar l}  g_{k\bar q} -F_{k\bar l} g_{p\bar q})+2 \left(F_p \wh \nabla^{(i)}_{\bar l}  g_{k\bar q}-F_k \wh\nabla^{(i)}_{\bar l}  g_{p\bar q} \right)\\
&\quad +2e^{2F} F_{\bar l} (T_{\tilde g})_{pk\bar q}+e^{2F}\wh\nabla^{(i)}_{\bar l} (T_{\tilde g})_{pk\bar q}\\
&=2\rho_0^{-1}\mathfrak{F}'(\rho_{p\bar l} g_{k\bar q} -\rho_{k\bar l} g_{p\bar q})
+2\rho_0^{-2}\mathfrak{F}''(\rho_{p}\rho_{\bar l} g_{k\bar q}-\rho_{k}\rho_{\bar l} g_{p\bar q})\\
&\quad +2\mathfrak{F}'\rho_0^{-1} \left(\rho_p \wh \nabla^{(i)}_{\bar l}  g_{k\bar q}-\rho_k \wh\nabla^{(i)}_{\bar l}  g_{p\bar q} \right)\\
&\quad +2e^{2F}\rho_0^{-1} \mathfrak{F}' \rho_{\bar l} (T_{\tilde g})_{pk\bar q}+e^{2F}\wh\nabla_{\bar l} (T_{\tilde g})_{pk\bar q}
\end{split}
\end{equation}
Using the fact that
$$(\wh\Gamma^{(i)} -\wh\Gamma)_{pq}^l=2F_p \delta^l_q=2\rho_0^{-1}\mathfrak{F}'\rho_p \delta^l_q$$
and hence
\begin{equation}
\begin{split}
\wh\nabla^{(i)}_{\bar l} g_{k\bar q}&=-2F_{\bar l} g_{k\bar q}+(\wh\nabla_{\bar l}^{(i)}-\wh\nabla_{\bar l})g_{k\bar q}+e^{2F} \wh\nabla_{\bar l} \tilde g_{k\bar q}\\
&=-2\rho_0^{-1}\mathfrak{F}'\rho_{\bar l} g_{k\bar q}+2\rho_0^{-1}\mathfrak{F}'\rho_{\bar l} g_{k\bar q}+e^{2F} \wh\nabla_{\bar l} \tilde g_{k\bar q}
\end{split}
\end{equation}
We may further infer that (iii) is true.

Now we examine the holomorphic sectional curvature after conformal change. Let $e_1\in T^{1,0}U_R$ be such that $|e_1|_{h_i}=1$, $|e_1|_h =e^{-F}$. Let $\kappa(x)$ be the  upper bound of the holomorphic sectional curvature of $h$ at $x$
\begin{equation}
\begin{split}
\hat R_{1\bar11\bar1}
&= -\partial_1 \partial_{\bar 1} (e^{2F}h_{1\bar 1}) +e^{-2F}h^{p\bar l} \partial_1 (e^{2F}h_{1\bar l})\cdot \partial_{\bar 1} (e^{2F}h_{p\bar 1})\\
&=-\partial_1 (e^{2F} \partial_{\bar 1} h_{1\bar 1}+2e^{2F} h_{1\bar 1}F_{\bar 1})\\
&\quad
+ e^{-2F} h^{p\bar l}\left(e^{2F} \partial_1h_{1\bar l} +2\hat h_{1\bar l}F_1 \right)\left(e^{2F}\partial_{\bar 1} h_{p\bar 1}+2\hat h_{p\bar 1} F_{\bar 1}\right)\\
&=e^{2F} \tilde R_{1\bar11\bar1}-2\hat h_{1\bar1} F_{1\bar1}
\\
&\leq e^{-2F}\kappa -2F_{1\bar1}\\
&\leq e^{-2F}\kappa+c_3(\beta+\beta^2) .
\end{split}
\end{equation}
Estimate $|\wh\nabla^{(i)}_{\bar\partial} T_i|_{h_i}$ in a similar way as above, we may conclude that
$$
\frac{n+1}{2n}\kappa_i(x)+|\wh\nabla^{(i)}_{\dbar} T_i |_{h_i}\le e^{-2F}\kappa_0
+c_3 \beta(1+ \beta ).
$$
From this (iv) is true.

\end{proof}

Now we are able to construct a solution of the Chern Ricci flow on $M$.
\begin{proof}[Proof of Theorem \ref{t-existence}]
For each sufficiently large $\rho_i$, $(U_{\rho_i},g_{0,i})$ has bounded geometry by Lemma \ref{l-boundedgeom}. By Lemma \ref{conformal-t-h}, using the notation in the lemma, we have for any $\e>0$:

$$
\frac{n+1}{2n}\kappa_i(x)+|\bar\nabla_{\dbar} T_i |_{h_i}\le \kappa_0
+c(\beta+(1+\e^{-1})\beta^2)+\e \gamma=:\kappa_{0,i}.
 $$
 $$
 \mathfrak{s}_i=\sup_M\lf(|T_{0,i}|^2_{h_i}+|T_{0,i}|_{h_i}|\wh T_i|_{h_i}+|\wh \nabla^{(i)}_{\dbar}T_{0,i}|_{h_i}+\kappa_{0,i}\ri)
 $$
 Then
 \bee
 \mathfrak{s}_i\le \kappa_0+c_1\beta(1+\beta) =:\mathfrak{s}.
 \eee
 By Lemma \ref{l-bgeom-1}, there is a solution $g_i(t)$ on $U_{\rho_i}\times[0,S)$ with initial metric $g_{0,i}$ where
 $$
 S=\frac1{3c_1(n\a+1)^3\mathfrak{s}}
 $$
 for some constant $c_1=c_1(n)$. Moreover, $g_i$ is uniformly equivalent to $g_{0,i}$ and
 \be\label{e-tr-1}
 \tr_{g_i}h_i\le v(t)-1
 \ee
 on $U_{\rho_i}\times[0,S)$
where
$$
v(t)= \lf(\frac1{(n\a+1)^{-3}-3c_1 \mathfrak{s}t}\ri)^\frac13.
$$

Fix any compact subset $K\subset M$. Then for sufficiently large $i$, $g_i(t)$ is a solution of the Chern Ricci flow defined on $U_{\rho_i}\supset U_{2r}\supset U_r\supset K$ for some large $r>0$. By Lemma \ref{locR}, for any $(x,t)\in K\times [0,S]$,
$$R_{g_i(t)} \geq -\max\left\{C_n \beta^{-1}  [(r_i-r)^{-2}+1],\sup_{\rho(y)<2r}R_-(y,0) \right\}$$
where we have used the fact that $g_i(0)=g_0$ on $U_r$ for sufficiently large $\rho_i$. In particular, it is bounded from below uniformly. Since
\begin{align*}
\frac{\partial}{\partial t} \left(\log \frac{\det g_i(t)}{\det h}\right)
&=-R_{g_i(t)}\leq C(n,K,\alpha).
\end{align*}
Therefore, on $K\times [0,S]$,
$$\alpha h\leq g(t)\leq e^{C t}\beta^{-n}h.$$
By the local estimate of the Chern Ricci flow \cite{ShermanWeinkove2013}, for any $k\in \mathbb{N}$, there is $C(n,k,g_0,h,\beta)$ such that for any $(x,t)\in K\times [0,S]$,
$$|\hat\nabla^k g_i(t)|_h\leq C(n,k,g_0,h,\beta).$$
By taking diagonal subsequence and using Arzel\`a-Ascoli theorem, we may obtain a limiting solution of $g(t)$ defined on $M\times [0,S)$. The conclusion on $\tr_gh$ follows from  \eqref{e-tr-1}. This completes the proof of the theorem.
\end{proof}

Next we want to apply Theorem \ref{t-existence} to obtain longtime solution for \KR flow.
\begin{thm}\label{t-longtimeKRF}
Let $(M,g_0)$ be a complete \K manifold and $h$ is a fixed complete Hermitian metric on $M$ such that the following holds.
\begin{enumerate}
\item[  (i)] There exists smooth exhaustion $\rho\geq 1$ such that
$$\limsup_{\rho\rightarrow \infty}\left(\frac{|\partial \rho|_h}{\rho}+\frac{|\sqrt{-1}\partial\bar\partial \rho|_h}{\rho} \right)=0;$$
\item[\bf (ii)]  and the holomorphic sectional curvature of $h$ and torsion of $h$ satisfy
$$H_h+\frac{2n}{n+1}|\wh\nabla_{\bar\partial} \wh T|_h\leq 0.$$
\item[ (iii)]  $\exists \a>1$ such that on $M$,
$\a^{-1}g_0\leq h\leq \a g_0$, $|\wh T|_h\leq \a$;
\item [(iv)]
$\limsup_{\rho\rightarrow \infty}\rho^{-1}|\wh\nabla g_0|_h =0;$

\end{enumerate}
Then there is $\beta(n,\a)>0$ such that the \K Ricci flow has a complete solution $g(t)$ on $M\times [0,+\infty)$ with $g(0)=g_0$ and satisfies
 $$\beta h\leq g(t)$$ on $M\times [0,+\infty)$.
\end{thm}
\begin{proof}
By Theorem \ref{t-existence} and the assumptions, one can apply the theorem with $\beta$ arbitrarily small. Hence one can find solution $g_i(t)$ to the Chern-Ricci flow with with $g_i(0)=g_0$ on $M\times[0,T_i]$ with $T_i\to\infty$. Moreover, $\tr_gh\le c(n,\a)$. Using  the local estimate of scalar curvature in Lemma \ref{locR} as in the proof of Theorem \ref{t-existence}, the results follow.
\end{proof}

\section{Existence  of the \K Einstein metric}\label{LTC}

{ In this section, we discuss the existence  the \K Einstein metric on $M$ via the \KR flow. We have the following:}

\begin{thm}\label{main}
Suppose there is a complete Hermitian metric $h$ on $M$ compatible with the complex structure $J$ such that
\begin{align}\label{Hneg}
H_h+\frac{2n}{n+1}|\wh\nabla_{\dbar} \wh T|_h\leq -k
\end{align}
for some $k>0$. Then any longtime solution of normalized \KR flow $g(t)$ will converge uniformly in $C^\infty$ to a \K Einstein metric  $g_\infty=-\Ric(g_\infty)$. {In particular, there is no Ricci flat \K metric on $M$ compatible with the same complex structure $J$.}
\end{thm}

{ Combining with Theorem \ref{t-longtimeKRF}, we have   the following results which are some generalization of    the result by Wu-Yau \cite{WuYau2017}:}

\begin{cor}\label{c-Wu-Yau} Let $(M^n,g_0)$ be a complete noncompact \K manifold with holomorphic sectional curvature bounded above by a  negative constant. Suppose $M$  supports an exhaustion function with uniformly bounded gradient and uniformly bounded Hessian, which is the case if $M$  has bounded curvature. Then $M^n$ support a unique K\"ahler Einstein metric with negative scalar curvature.
\end{cor}

\begin{cor}\label{c-Xu} Let $(M^n,g_0)$ be a complete noncompact \K manifold with complex dimension $n$. Suppose there is $K_1, r,A_0,r>0,p>n$ such that for all $x\in M$, $f\in C^\infty_0 (B_{g_0}(x,4r))$,
\begin{equation}
\left\{ \begin{array}{r c l}
\displaystyle H_{g_0} &\le\kappa_0< 0;\\
   \displaystyle {\avint_{B_{g_0}(x,r)}|Rm(g_0)|^p \, d\mu_{g_0}}&\leq& K_1;\\
    \displaystyle\avint_{B_{g_0}(x,4r)}|f|^\frac{2n}{n-1}\, d\mu_{g_0}&\leq& A_0 r^2 \avint_{B_{g_0}(x,4r)}|\nabla f|^2\, d\mu_{g_0}.
\end{array}
\right.
\end{equation}
Then $M$ supports a unique K\"ahler Einstein metric with negative scalar curvature.
\end{cor}
\begin{proof}
By \cite{Xu2013}, there is a complete short time  solution $g(t)$ to the Ricci flow with $g(t)=g_0$ such that $|\Rm(g(t))|\le Ct^{-a}$ for some $0<a<1$. By \cite{HuangTam2015}, $g(t)$ is \K for $t>0$. On the other hand for fixed $g(t)$, there is an exhaustion function $\rho$ with uniformly bounded gradient and uniformly bounded Hessian. Since $g(t)$ is uniformly equivalent to $g_0$, $\rho$ is also an exhaustion function $\rho$ with uniformly bounded gradient and uniformly bounded Hessian with respect to $g_0$. By Corollary \ref{c-Wu-Yau}, the result follows.
\end{proof}

We also want to discuss metrics which are uniformly equivalent to $g_0$ as in the previous corollaries. The following is an immediate consequence of Theorem \ref{t-longtimeKRF} and Theorem \ref{main}.

\begin{cor}\label{c-uni-equivalent} Let $(M^n,h)$ be a complete noncompact Hermitian manifold { with bounded torsion} and satisfies the conditions in Corollary  \ref{c-Wu-Yau}: The holomorphic sectional curvature of $h$ is bounded above by a negative constant, and $(M,h)$ support an exhaustion function with uniformly bounded gradient and uniformly bounded Hessian.  Let $g_0$ be a \K metric which is uniformly  equivalent to $h$ so that $|\wh \nabla g_0|_h$ is bounded, where $\wh \nabla$ is the derivative with respect to the Chern connection of $h$. Then $M$ supports a unique K\"ahler-Einstein metric with negative scalar curvature.
\end{cor}



Let us prove Theorem \ref{main}. Here we do not have a good exhaustion function for $h$. However, the distance function $d(x,t)$ in a Ricci flow behaves well. Using the idea in the work of Chen \cite{Chen2009}, we have the following:

\begin{lma}\label{l-Chen}
Let $(M^m,g(t))$ be a complete noncompact  solution to the Ricci flow on $M\times[0,T]$ with $0<T<\infty,$  where $m\ge 2$ is the real dimension of $M$. Let $Q$ be a   smooth function so that
$$
\heat Q\le -\a Q^2+\beta
$$
for some $\a, \beta>0$ at the point where $Q>0$. Then
$$
tQ(x,t)\le \frac{1+\sqrt{1+4\a\beta T^2}}{2\a}.
$$
on $M\times(0,T]$.
\end{lma}
\begin{proof} Let $x_0\in M$, and let $r_0>0$ be small enough so that:

$$\Ric(x,t)\leq (m-1)r_0^{-2}$$
for $x\in B_t(x_0,r_0)$, $t\in [0,T]$.  By \cite{Perelman} (see also \cite{Chen2009}), we then have
\begin{align}
\heat d_t(x,x_0) \geq -\frac{5(m-1)}{3} r_0^{-1}
\end{align}
whenever $d_t(x,x_0)\ge r_0$ in the sense of barrier, where $d_t(x,x_0)$ is the distance function from $x_0$ with respect to $g(t)$. In the followig, argue as in \cite{HuangTam2015}, we may assume that $d_t(x,x_0)$ to be smooth when applying maximum principle. We consider the function
$$u(x,t)=t\varphi\left(\frac1{Ar_0}\lf[ d_t(x,x_0)+\frac{5(m-1)t}{3r_0}  \ri] \right) Q(x,t) $$
 $A$ is sufficiently large so that $Ar_0>>\frac{5(m-1)T}{3r_0}  $,  and $\varphi$ is a fixed smooth nonnegative non-increasing function such that $\varphi\equiv 1$ on
$(-\infty, \frac{1}{2}]$, vanishes outside $[0,1]$ and satisfies $|2\frac{(\varphi')^2}{\varphi}|+|\varphi''|\leq c_1$ for some absolute constant.
 Note that $u$ also depends on $A$. However,
 $$
 u(x_0,t)=tQ(x_0,t).
 $$
 if $Ar_0\ge  \frac{10(m-1)T}{3r_0}.$

If $u\le 0$, then we are done.  Suppose the function $u>0$ somewhere, then there exists $(x_1,t_1)$ with $0<t_1\le T$ so that $u$ attains its maximum at $(x_1,t_1)$. We have at $(x_1,t_1)$ we have
 \bee
0\le  \heat u; \ \ \nabla Q=-\frac{\nabla \phi}{\phi}Q.
 \eee
 Suppose $d_{t_1}(x_1,x_0)<r_0$, then $u(x,t)=tQ(x,t)$ near $x_1,t_1$ provided $Ar_0$ is large enough. Then we have at $(x_1,t_1)$
 \bee
 \begin{split}
 0\le &\heat u\\
 =& t_1\heat Q+Q\\
 \le &-\a t_1Q^2+\beta t_1+Q
 \end{split}
\eee
and so
\bee
0\le- \a u^2+u+\beta T^2
\eee
which implies
\be\label{e-u-1}
u(x_0,t)\le u(x_1,t_1)\le \frac{1+\sqrt{1+4\a\beta T^2}}{2\a}.
\ee
for $t\in [0,T]$.

Suppose $d_{t_1}(x_1,x_0)\ge r_0$, then  at $(x_1,t_1)$,

\begin{equation*}
\begin{split}
0&\leq   \heat u\\
&=Qt \heat \varphi+\varphi \heat (Qt)  -2t\langle \nabla \varphi,\nabla Q\rangle\\
&\leq Qt \varphi'\frac{1}{Ar_0} \left[ \heat d_t(x,p)+\frac{5}{3} (m-1) r_0^{-1} \right]\\
&\quad  +|\varphi''|\frac{1}{(Ar_0)^2} tQ+\varphi \left( -\a t Q^2+\beta t+Q\right) +2tQ\frac{1}{(Ar_0)^2}\cdot \frac{(\phi')^2}{\phi} \\
&\leq   -\a t\varphi Q^2+\varphi Q+ \beta t\varphi  +c_1Qt\frac{1}{(Ar_0)^2}.
\end{split}
\end{equation*}
Multiply both the inequality by $t\phi=t_1\phi$, we have
$$
0\le -\a u^2+\lf(1+\frac{c_1T}{Ar_0^2}\ri)u+\beta T^2.
$$
Hence we have
\be\label{e-u-2}
u(x_0,t)\le u(x_1,t_1)\le \frac{1+\frac{c_2T}{Ar_0^2}+\sqrt{\lf(1+\frac{c_1T}{Ar_0^2}\ri)^2+4\a\beta T^2}}{2\a}.
\ee
for $0<t\le T$. Let $A\to\infty$ together with \eqref{e-u-2} and the fact that $x_0$ is any point in $M$, we conclude the lemma is true.
\end{proof}

{As an application of the lemma, we prove   that complete noncompact  \K Einstein metrics   with negative scalar curvature  is unique up to scaling. Here we do not assume the curvature is bounded, see also \cite{WuYau2017}. Namely, we have the following:
Moreover, the \K Einstein metric with negative scalar curvature must be unique up to scaling.
\begin{thm}
Suppose $\omega_1$ and $\omega_2$ are complete noncompact \K Einstein metrics on $M$ with
$\Ric(\omega_i)=-\omega_i$ for $i=1,2$. Then $\omega_1=\omega_2$ on $M$.
\end{thm}

\begin{proof} Let 
Let $\wt\omega_1(t)=(t+1)\omega_1$ and $\wt\omega_2(t)= (t+1)\omega_2$. Then both $\wt\omega_1, \wt\omega_2$ are solutions to the  \KR flow on $M\times [0,+\infty)$. Define $F(x,t)=F(x)$ to be the function
$$F(x,t)=\log \left[\frac{\wt\omega_2^n}{\wt\omega_1^n}\right]^\frac{1}{n}=\log \left[\frac{ \omega_2^n}{\omega_1^n}\right]^\frac{1}{n} $$
which is independent of $t$.
The function $F$ is independent of $t>0$ but we treat it as a function over the \KR flow. Let $\Delta$ be the Laplacian of $\wt\omega_1$. Then on $M\times(0,\infty)$,  On $[0,1]$, it satisfies
\begin{align*}
\heat F&=\frac{1}{t+1}\left(1-\frac{1}{n}\tr_{\omega_1}\omega_2\right)\\
&\leq \frac{1}{t+1}\left(1-e^F\right)\\
&\leq - \frac{1}{4}F^2
\end{align*}
whenever $F>0$ on $M\times[0,1]$.  

Apply Lemma \ref{l-Chen} on $M\times [0,1]$, $tF \leq 8$. In particular, $F(x)$ is bounded from above uniformly on $M$. By interchanging $\omega_1$ and $\omega_2$, we conclude that $F$ is a bounded function on $M$. Let $\Delta_1$ be the Laplacian of $\omega_1$, we have as above
$$
-\Delta_1F\le 1-e^F.
$$
By the generalized maximum principle \cite{ChengYau}, we conclude that $F\le0$. Interchanging the roles of   $\omega_1$ and   $ \omega_2$, we can prove similarly that $F\ge0$. Hence $F=0$. So $\ddbar F=0$ and $\omega_1=\omega_2$ because they are K\"ahler-Einstein.

\end{proof}

}

Now we are ready to prove Theorem \ref{main}.

\begin{proof}[Proof of Theorem \ref{main}]
Let $g(t)$ and $h$ as in the Theorem. Let $\Lambda=\tr_gh$. By Corollary \ref{c-trace}, we have
$$
\heat \Lambda\le -c_1k\Lambda^2,
$$
for some $c_1$ depending only on $n$. By Lemma \ref{l-Chen} with $\beta=0$, we conclude that
\be\label{e-KE-1}
\Lambda(x,t)\le \frac1{2c_1kt}
\ee
on $M\times(0,\infty)$.

On the other hand, let $R(x,t)$ be the scalar curvature of $g(t)$ at $x$ and let $R_-(x,t)$ be its negative part. For any $\e>0$, let $f=\frac12\lf((R^2+\e^2)^\frac12-R\ri)$. Note that if $\e\to0$, then $f\to  R_-$. Using the fact that
\bee
\heat R\ge \frac1n R^2,
\eee
 direct computations show that
\bee
\heat f\le -\frac1nf(f-2c_1\e) \le \frac1n(f-c_1\e)^2+c_2\e
\eee
for some absolute constant $c_1>0$ and $c_2>0$ depending only on $n$. By Lemma \ref{l-Chen}, we conclude that
\bee
t(f-c_1\e)\le \frac n2\lf(1+\sqrt{1+\frac{4c_2\e}n}\ri).
\eee
on $M\times(0,\infty)$. Let $\e\to0$, we conclude that
\be\label{e-KE-2}
tR(x,t)\ge -n.
\ee
Since
$$
\frac{\p }{\p t}\log\lf(\frac{\det g(t)}{\det h}\ri)=-R\le \frac nt
$$
we conclude for any bounded open  set $\Omega$, there is a constant $C_1$ depending only on $\Omega, g(1), h, k, n$. such that
$$
 \frac{\det g(t)}{\det h} \le C_1t^n
$$
on $\Omega\times[1,\infty)$. Combining this with \eqref{e-KE-1}, we conclude that
$$
C_2^{-1}th\le g\le C_2th
$$
on $\Omega\times[1,\infty)$ for some constant $C_2>0$ depending only on $\Omega, g(1), h, k, n$.

Consider the normalized metric
$$
\wt g(x,s)=e^{-s}g(x,e^s).
$$
Then we have
\be\label{e-NKRF}
\frac{\p}{\p s}\wt g =-\Ric(\wt g)-\wt g
\ee
on $M\times[0,\infty)$, and
\be\label{e-KE-3}
C_2^{-1}h\le \wt g(s)\le C_2 h
\ee
on $\Omega\times[0,\infty)$.

In the following, let $\omega(t)$, $\wt\omega(s)$ be the \K forms of $g(t)$, $\wt g(s)$ respectively. By \cite[Theorem 2.17]{SongWeinkove2013}, we conclude that for any bounded open set in $M$ and $\ell\ge0$,  there is a constant $C_3$ depending only on $\Omega, g(1), h, k, n$. and $\ell$ such that
\be\label{e-KE-4}
||\wt\omega(s)||_{C^{\ell}(\Omega, \wt g_0)}\le C_3.
\ee

On the other hand, let
\bee
\wt\phi(x,s)=e^{-s}\int_0^s e^\tau\log\lf(\frac{ (\wt\omega(\tau))^n}{(\wt\omega(0))^n}\ri)d\tau.
\eee
Then
\be\label{e-tphi-1}
\wt\omega(s)=e^{-s}\wt\omega(0)-(1-e^{-s})\Ric(\wt\omega(0))+\ii\ddbar\wt \phi(s).
\ee
Moreover,
\be\label{e-tphi-2}
\left\{
  \begin{array}{ll}
   \frac{\p}{\p s} \wt\phi =  &\log\lf(\frac{ \wt\omega^n}{(\wt\omega(0))^n}\ri)-\wt\phi \hbox{\ \ in $M\times[0,\infty)$;} \\
    \wt\phi(0)=   &0 \hbox{\ \ on $M$.}
  \end{array}
\right.
\ee
Denote $\p_s\wt\phi$ by $\wt \phi'$ etc., and let $\wt R$ be the scalar curvature of $\wt g$, then
\be\label{e-tphi-3}
\begin{split}
\wt\phi''+\wt\phi'=&-\wt R-n\\
=&-e^s R(g(e^s))-n\\
=&-e^s\lf(R(g(e^s))+e^{-s}n\ri)\\
\le &0
\end{split}
\ee
by \eqref{e-KE-2}.
Hence $\wt\phi'+\wt\phi$ is non-increasing and $\wt\phi'+\wt\phi\le0$  because $\wt\phi'+\wt\phi=0$ at $s=0$. On the other hand, by \eqref{e-KE-4} and \eqref{e-tphi-2}, we conclude that for any bounded open set $\Omega$, there exists $s_i\to\infty$ such that
$$
(\wt\phi'+\wt\phi)(s_i)
$$
converge uniformly in $C^\infty$ norm in $\Omega$. By the monotonicity of $\wt\phi'+\wt\phi$, we conclude that $\wt\phi'+\wt\phi$ converges in $C^\infty$ norm in $\Omega$ to some function.

By \eqref{e-tphi-3}, we have
$$
(e^s\wt\phi)'\le 0,
$$
and so $\wt\phi'\le0$ because $\phi'=0$ at $s=0$. Combine this with \eqref{e-KE-4} and \eqref{e-tphi-1}, we conclude that $\wt\phi$ also converge in $C^\infty$ norm to some function $\wt\phi_\infty$. Hence $\wt\phi'$ also converge in $C^\infty$ norm to some function. However, by \eqref{e-KE-3} we conclude that $\phi$ is bounded from below. This implies that $\wt\phi'\to0$ as $s\to\infty$. Moreover, $\wt \omega(s)\to \wt \omega_\infty$ in $C^\infty$ norm in $\Omega$ as $s\to\infty$ with
$$
\wt\omega_\infty=-\Ric(\wt\omega(0))+\ii\ddbar \wt\phi_\infty.
$$
Note that $\wt\omega_\infty$ is a \K form of a \K metric by \eqref{e-KE-3}. Moreover,
$$
\wt\phi_\infty=\log\lf(\frac{ \wt\omega^n}{(\wt\omega(0))^n}\ri).
$$
Taking $\ddbar$ to both sides, we conclude that
$$
\Ric(\wt\omega_\infty)=-\wt\omega_\infty.
$$

 Suppose $\bar \omega$ is a Ricci flat metric compatible with the same complex structure of $h$. Then $ \omega(t)=\bar\omega$ is a steady solution of the \KR flow. By the convergence of normalized \KR flow, $t^{-1}\omega(t)$ converges to a K\"ahler Einstein metric on $M$ which is impossible since $t^{-1}\omega(t)\equiv t^{-1}\bar\omega$ converges to a zero tensor on $M$. This completes the proof.
\end{proof}

\end{document}